\documentclass[10pt,a4,reqno]{amsart}

\usepackage{amsmath}
\usepackage{mathptmx} 
\usepackage{amssymb}
\usepackage{amsthm}
\usepackage{txfonts}
\usepackage{textcomp}
\usepackage{microtype}
\usepackage[usenames,dvipsnames]{xcolor}
\usepackage{enumerate}
\usepackage{dsfont}
\usepackage{hyperref}
\usepackage[utf8]{inputenc}
\usepackage{mathrsfs}
\usepackage{graphicx}
\usepackage{subfig}
\theoremstyle{definition}
\newtheorem{Def}{Definition}[section]

\theoremstyle{plain}

\newtheorem{Lem}[Def]{Lemma}
\newtheorem{Cor}[Def]{Corollary}
\newtheorem{Thm}[Def]{Theorem}

\newcommand{\e}{\mathrm{e}}

\def\N{\mathbb N}
\def\R{\mathbb{R}}

\title{A note on measure-geometric Laplacians}

\author{M.~Kesseb\"ohmer, T.~Samuel and H.~Weyer}

\address{Fachbereich 3 Mathematik, Universit\"at Bremen, Bibliothekstr. 1, 28359 Bremen, Germany}
\email{mhk@math.uni-bremen.de, tony@math.uni-bremen.de, hendrik@math.uni-bremen.de}

\date{\today}

\begin{document}

\maketitle

\begin{abstract}
We consider the measure-geometric Laplacians $\Delta^{\mu}$ with respect to atomless compactly supported Borel probability measures $\mu$ as introduced by Freiberg and Z\"ahle in 2002 and show that the harmonic calculus of $\Delta^{\mu}$ can be deduced from the classical (weak) Laplacian.  We explicitly calculate the eigenvalues and eigenfunctions of $\Delta^{\mu}$.  Further, it is shown that there exists a measure-geometric Laplacian whose eigenfunctions are the Chebyshev polynomials and illustrate our results through specific examples of fractal measures, namely Salem and inhomogeneous self-similar Cantor measures.
\end{abstract}

\section{Introduction}

Kigami constructed a Laplacian for fractal sets, in particular post-critically finite attractors of iterated function systems, as the limit of (normalised) difference operators on a sequence of finite graphs, which approximate the attractor, see for instance \cite{Ki01,St06}.  (Key examples of such attractors are the closed unit interval and the Sierpinski gasket.)  This theory has been further developed and extensively studied by Kigami, Kusuoka, Rogers, Strichartz and Teplayev to name but a few.  Another approach was developed by Goldstein \cite{Go87}, Kusuoka \cite{Ku87}, Barlow and Perkins \cite{BP88} and Lindstr{\o}m \cite{Li90} and has a probabilistic motivation. This approach defines a Laplacian as the infinitesimal generator of a Brownian motion, but certain assumptions on the class of attractors need to be made.  Again this approach uses finite graph approximation methods.  A different probabilistic approach has also been considered by Denker and Sato \cite{DS01,DS99} and Ju, Lau and Wang \cite{Lau12}.

Given an atomless Borel probability measure $\mu$ supported on a compact subset $K$ of $\R$, Freiberg and Z\"{a}hle \cite{ZF02} introduced a measure-geometric approach to define a (second order) differential operator $\Delta^{\mu}$.  This approach is motivated by the fundamental theorem of calculus and based on the works of Feller \cite{Fe57} and Kac and Kre\u{\i}n \cite{KK68}. In the case that $K$ is a fractal set, whereas the previous approaches relied on graph approximation methods, with this new approach such approximations are not required.  Indeed in \cite{ZF02} a harmonic calculus of $\Delta^{\mu}$ is developed.  In the special case when $\mu$ is a self-similar measure supported on a Cantor-like set, the authors prove that the eigenvalue counting function of $\Delta^{\mu}$ is comparable to the square-root function.

In this article we extend the results of \cite{ZF02} in that we give a complete solution to the eigenvalue problem.  Namely, for a given atomless Borel probability measure $\mu$ supported on $K$ where $\{ a, b\} \subset K \subset[a, b]$, for some real numbers $a < b$, with continuous distribution function $F_{\mu}$, we prove the following Theorem.  Indeed an indication that such a result should hold is given in \cite{Z05}.

\pagebreak

\begin{Thm}\label{thm:ev and ef general continuous}
Let the domain of the operator $\Delta^{\mu}$ be denoted by $\mathscr{D}_{2}^{\mu}$ and set $\lambda_{n} \coloneqq -(\pi n)^{2}$, for $n \in \N_{0}$.
\begin{enumerate}[(i)]
\item\label{thm:ev and ef general continuous_i} The eigenvalues of $\Delta^{\mu}$ on $\mathscr{D}_{2}^{\mu}$ under homogeneous Dirichlet boundary conditions are $\lambda_{n}$, for $n \in \N$, with corresponding eigenfunctions
\begin{align*}
f_{n}^{\mu}(x)\coloneqq \sin(\pi n \cdot F_{\mu}(x)), \quad \text{for} \; x\in [a, b].
\end{align*}
\item\label{thm:ev and ef general continuous_ii} The eigenvalues of $\Delta^{\mu}$ on $\mathscr{D}_{2}^{\mu}$ under homogeneous von Neumann boundary conditions are $\lambda_{n}$, for $n \in \N_{0}$, with corresponding eigenfunctions
\begin{align*}
g_{n}^{\mu}(x)\coloneqq \cos(\pi n \cdot F_{\mu}(x)),  \quad \text{for} \; x\in [a, b].
\end{align*}
\end{enumerate}
\end{Thm}

 In the proof of this result we will see that apart from the constant functions the $\mu$-derivative of  von Neumann eigenfunctions  are Dirichlet eigenfunctions and vice versa. 
  
Moreover,  Theorem \ref{thm:ev and ef general continuous} demonstrates that the concept of fractal Fourier analysis (as defined by Bandt, Barnsley, Hegland and Vince \cite{BBHV14}) has interesting connections to the measure-geometric Laplacian $\Delta^{\mu}$ considered here.

Letting $N^{\mu}_{D} \colon \R \to \N$ denote the eigenvalue counting function of $-\Delta^{\mu}$ on $\mathscr{D}_{2}^{\mu}$ under homogeneous Dirichlet boundary conditions and letting $N^{\mu}_{N} \colon \R \to \N$ denote the eigenvalue counting function of $-\Delta^{\mu}$ on $\mathscr{D}_{2}^{\mu}$ under homogeneous von Neumann boundary conditions, the following can be obtained.

\begin{Cor}
\begin{align*}
\lim_{x \to \infty} \frac{N^{\mu}_{D}(x)}{(\pi x)^{1/2}} = \lim_{x \to \infty} \frac{N^{\mu}_{N}(x)}{(\pi x)^{1/2}} = 1
\end{align*}
\end{Cor}

To illustrate our results, we show that there exists a measure-geometric Laplacian whose eigenfunctions are the Chebyshev polynomials and also consider fractal measures, namely Salem and inhomogeneous self-similar measures.

\subsection*{Outline} 

In Section \ref{sec:notation} we present the necessary definitions and basic properties.  In Section \ref{sec:main_result} we give the proof Theorem \ref{thm:ev and ef general continuous}; first proving the result for when $\mu$ is the one-dimensional Lebesgue measure $\Lambda$ restricted to the unit interval, and then using a transfer principle to prove the result for the general case.  We conclude this article by considering specific examples, namely when $F_{\mu}(x) = \arccos(-x)/\pi$, when $F_{\mu}$ is a Salem function, and when $\mu$ is an inhomogeneous self-similar measure (that is a measure singular to the Lebesgue measure).

\section{Preliminaries: Derivatives and the Laplacian with respect to measures}\label{sec:notation}

Following conventions, the natural numbers will be denoted by $\N$, the non-negative integers by $\N_{0}$ and the real numbers by $\R$.  Throughout let $a < b$ be two fixed real numbers.  The one-dimensional Lebesgue measure will be denoted by $\Lambda$.  Henceforth, we will be concerned with an atomless Borel probability measure $\mu$ supported on $K$ where $\{ a, b\} \subset K \subset[a, b]$.  We denote the set of $\R$-valued functions with domain $[a, b]$ which are square integrable with respect to $\mu$ by $\mathfrak{L}^{2}(\mu)$. We set $L^{2}(\mu)$ to be the space of (measure) equivalence classes of $\mathfrak{L}^{2}(\mu)$-functions with respect to $\mu$.  For a measurable function $f$ we write $f \in L^{2}(\mu)$ when there exists an equivalence class of $L^{2}(\mu)$ to which $f$ belongs.  For a given set $A \subseteq \R$, the characteristic function of $A$ is denoted by $\mathds{1}_{A}$.

We recall the definition of the operator $\Delta^{\mu}$ as given in \cite{ZF02}.  We define $\mathscr{D}_{1}^{\mu} \subseteq\mathfrak{L}^{2}(\mu)$ by
\begin{align}\label{eq:D1}
\mathscr{D}_{1}^{\mu} \coloneqq \left\{f\in\mathfrak{L}^2(\mu) \colon\! \exists f' \in L^ 2(\mu) \, \text{s.t.} \,
f(x) = f(a) + \int \mathds{1}_{[a,x]}(y) \cdot f' (y) \, \mathrm{d}\mu(y),
\, x \in [a, b] \right\}.
\end{align}

For $f \in \mathscr{D}_{1}^{\mu}$ the function $f' \in L^{2}(\mu)$ given in \eqref{eq:D1} is unique.  Moreover, every function in $\mathscr{D}_{1}^{\mu}$ is continuous on $[a, b]$.  (For a proof of the first statement see {\cite[Proposition 2.1]{ZF02}}; the second statement is a consequence of Lebesgue's dominated convergence theorem.)  Indeed from these observations, we conclude that if $f,g\in\mathscr{D}_{1}^{\mu}$ with $f\neq g$, then $\mu (\{x \in [a, b] \colon f(x)\neq g(x)\}) > 0$. Hence, there exists a natural embedding $\pi \colon \mathscr{D}_{1}^{\mu} \to L^{2}(\mu)$.  Unless otherwise stated, for both $\mathscr{D}_{1}^{\mu}$ and $\pi (\mathscr{D}_{1}^{\mu} )$ we write $\mathscr{D}_{1}^{\mu}$.  Letting $f\in\mathscr{D}_{1}^{\mu}$ and letting $f'$ be as in \eqref{eq:D1}, the operator
\begin{align*}
\nabla^{\mu}:\mathscr{D}_{1}^{\mu} &\to L^{2}(\mu),\; f \mapsto f',
\end{align*}
is called the \textit{$\mu$-derivative}.  We define $\mathscr{D}_{2}^{\mu} \subseteq \mathscr{D}_{1}^{\mu}$ by
\begin{equation*}
\mathscr{D}_{2}^{\mu} \coloneqq \left\{f\in\mathscr{D}_{1}^{\mu} \colon \nabla^{\mu}f\in\mathscr{D}_{1}^{\mu} \right\}.
\end{equation*}
The \textit{$\mu$-Laplace operator} is defined by
\begin{align*}
\Delta^{\mu} \colon \mathscr{D}_{2}^{\mu} &\to L^{2}_{\mu},\;
f \mapsto \nabla^{\mu} \circ \nabla^{\mu} (f)
\end{align*}
and is a non-positive operator \cite[Corollary 2.3]{ZF02}.

For $f \in \mathscr{D}_{2}^{\mu}$ Fubini's theorem gives the representation
\begin{align}\label{eq:Delta_Integral_eq}
f(x) = f(a) + \nabla^{\mu} f(a) \cdot F_{\mu}(x) + \int \mathds{1}_{[a, x]}(y) \cdot (F_{\mu}(x) - F_{\mu}(y)) \cdot \Delta^{\mu}f(y) \; \mathrm{d}\mu (y).
\end{align}
A function $f \in \mathscr{D}_{2}^{\mu}$ is said to satisfy \textit{homogenous Dirichlet boundary conditions} if and only if $f(a) = f(b) = 0$, whereas if $\nabla^{\mu}f(a) = \nabla^{\mu}f(b) = 0$, then we say that $f$ satisfies  \textit{homogenous von Neumann boundary conditions}.

From this  all \textit{$\mu$-harmonic functions} can be computed, namely functions for which $\Delta^{\mu}f\equiv 0$.  Indeed the set of $\mu$-harmonic functions is    a two-dimensional space and equal to 
\begin{align*}
\{x \mapsto A+ B \cdot F_{\mu}(x)\colon  A,B \in \R\}.
\end{align*}
Notice that the operator $\Delta^{\Lambda}$ is the classical weak Laplacian.

It is shown in the proof of \cite[Theorem 2.5]{ZF02} that if $l_{\kappa}$ is a solution to the eigenvlaue problem
\begin{align}\label{eq:eigenvalue_problem}
\Delta^{\mu}(l_{\kappa}) = \kappa \cdot l_{\kappa},
\end{align}
for some $\kappa \in \R$, under the boundary conditions $l_{\kappa}(a) = 0$ and $\nabla^{\mu} l_{\kappa}(a) = 1$, then it is unique and necessarily satisfies the Volterra type integral equation
\begin{align}\label{eq:integral equation lebesgue dirichlet}
l_{\kappa}(x) = F_{\mu}(x) + \kappa \int \mathds{1}_{[a, x]}(z) \cdot ( F_{\mu}(x) - F_{\mu}(z) ) \cdot  l_{\kappa}(z) \; \mathrm{d}\mu(z),\;x\in[a,b].
\end{align}
If $l_{\kappa}$ is a solution to the eigenvlaue problem \eqref{eq:eigenvalue_problem} under the boundary conditions $l_{\kappa}(a) = 1$ and $\nabla^{\mu} l_{\kappa}(a) = 0$, then it is unique and necessarily satisfies the Volterra type integral equation
\begin{align*}
l_{\kappa}(x) = 1 + \kappa \int \mathds{1}_{[a, x]} \cdot ( F_{\mu}(x) - F_{\mu}(z) )  \cdot l_{\kappa}(z) \; \mathrm{d}\mu(z),\;x\in[a,b].
\end{align*}

\section{Proof of Theorem \ref{thm:ev and ef general continuous}: Eigenfunctions and Eigenvalues}\label{sec:main_result}

Here we prove our main result Theorem \ref{thm:ev and ef general continuous}.  We divide the proof into two cases: the case when $\mu$ is the Lebesgue measure $\Lambda$ supported on the closed unit interval, and the general case when the distribution function $F_{\mu}$ of $\mu$ is continuous.  We will in fact deduce the general case from the specific case when $\mu = \Lambda$.

\subsection{The Lebesgue case}

\begin{Thm}\label{thm:lebesgue_case}  Let $\Lambda$ denote the Lebesgue measure restricted to the closed unit interval.
\begin{enumerate}[(i)]
\item Assuming homogeneous Dirichlet boundary conditions, the eigenvalues of $\Delta^{\Lambda}$ are $\lambda_{n} = -(\pi n)^{2}$, for $n \in \N$, with corresponding eigenfunctions $f_{n}^{\Lambda}(x) = \sin(\pi n \cdot x)$.
\item Assuming homogeneous von Neumann boundary conditions the eigenvalues of $\Delta^{\Lambda}$ are $\lambda_{n} = -(\pi n)^{2}$, for $n \in \N_{0}$, with corresponding eigenfunctions $g_{n}^{\Lambda}(x) = \cos(\pi n  \cdot x)$.
\end{enumerate}
\end{Thm}

\begin{proof}
On the set of continuous functions supported on $[0, 1]$ the $\Lambda$-Laplacian agrees with the (classical) weak Laplacian.  Using this observation in tandem with the classical Sturm-Liouville theory (see for instance \cite{GT12}) the result follows.
\end{proof}

\subsection{The general case}

From here on we will always assume that $\mu$ is a Borel probability measure supported on $K$, where $\{ a, b\} \subset K \subset[a, b]$, with continuous distribution function $F_{\mu} \colon [a, b] \to [0, 1]$.

\begin{Lem}\label{lem:equalities of measures}
We have that
\begin{enumerate}[(i)]
\item\label{lem:equalities of measures_i} $\mu \circ F_{\mu}^{-1}= \Lambda$ and 
\item\label{lem:equalities of measures_ii} $\Lambda \circ F_{\mu}= \mu$.
\end{enumerate}
\end{Lem}

\begin{proof}
It is sufficient to show the equality of the measures on a generator which is stable under intersection.  For this let $[a_{1},b_{1}]$ be an arbitrary closed subinterval of $[0,1]$.  For (\ref{lem:equalities of measures_i}) we set\linebreak $M \coloneqq \{ x \in [a,b] \colon a_{1} \leq F_{\mu}(x)\leq b_{1} \}$ and observe that
\begin{align*}
\mu(F_{\mu}^{-1}([a_{1},b_{1}])) = \mu( M ) = F_{\mu}(\sup \{ x \colon x\in M \} ) - F_{\mu}( \inf \{ x \colon x\in M \} ) = b_{1} - a_{1} = \Lambda([a_{1},b_{1}]).
\end{align*}
This yields the result stated in (i). For (\ref{lem:equalities of measures_ii}) we observe, for a subinterval $[a_{2}, b_{2}] \subseteq [a,b]$, that
\begin{align*}
\Lambda(F_{\mu}([a_{2}, b_{2}])) = \Lambda([F_{\mu}(a_{2}), F_{\mu}(b_{2})]) = F_{\mu}(b_{2}) - F_{\mu}(a_{2}) = \mu([a_{2},b_{2}]),
\end{align*}
yielding the required result.
\end{proof}

\begin{Lem}\label{lem:volterra}
Let $\kappa \in \R$ be given.  Under the boundary conditions $f(a) = A$ and $\nabla^{\mu} f (a) = B$, there exists a unique solution to the integral equation
\begin{align*}
f(x) = A + B \cdot F_{\mu}(x) + \kappa \int (F_{\mu}(x) - F_{\mu}(z)) \cdot f(z) \; \mathrm{d}\mu(z),\;x\in[a,b].
\end{align*}
\end{Lem}

\begin{proof}
For $\alpha > 0$, the set of continuous functions supported on $[a, b]$ equipped with the norm $\lVert f \rVert_{\alpha} \coloneqq \sup \{ \lvert f(x) \rvert \e^{-\alpha F_{\mu}(x)} \colon x \in [a, b] \}$ is complete.  For $\alpha > \kappa$ the operator
\begin{align*}
f(x) \mapsto A + B \cdot F_{\mu}(x) + \kappa \int (F_{\mu}(x) - F_{\mu}(z)) f(z) \; \mathrm{d}\mu
\end{align*}
is a contraction with respect to $\lVert \cdot \rVert_{\alpha}$.  The result follows from Banach fixed point theorem.
\end{proof}

We define the \textit{pseudoinverse} of the distribution function $F_{\mu}$ by
\begin{align*}
\check{F}_{\mu}^{-1} \colon [0, 1] &\to [a, b],\;
x \mapsto \inf \{ y \in [a, b] \colon F_{\mu}(y) \geq x \}.
\end{align*}
Notice that $F_{\mu} \circ \check{F}^{-1}_{\mu}(x) = x$, for all $x \in [0, 1]$, and that $\check{F}^{-1}_{\mu} \circ F_{\mu}(y) = y$, for $\mu$-almost all $y \in [a, b]$.  Also note that $\check{F}_{\mu}^{-1}(1) = b$ since $b$ belongs to the support of $\mu$.

We are now in a position to prove our main result, Theorem \ref{thm:ev and ef general continuous}.

\begin{proof}[Proof of Theorem \ref{thm:ev and ef general continuous}]
As the proofs under Dirichlet and von Neumann boundary conditions follow in an identical manner, we only present the proof of the result under Dirichlet boundary conditions, namely part (i). In essence, to obtain (ii) from (i) one only needs to interchange the functions $x \mapsto \sin(x)$ and $x \mapsto \cos(x)$ in what follows.

First we show that the functions $f_{n}^{\mu}$ are eigenfunctions.  Second we prove that if $l_{\kappa}$ is an eigenfunction of $\Delta^{\mu}$ with eigenvalue $\kappa$, then $l_{\kappa} \circ \check{F}^{-1}_{\mu}$ is an eigenfunction of $\Delta^{\Lambda}$ supported on $[0, 1]$.  Hence we conclude that $f_{n}^{\mu}$ are all of the eigenfunctions of $\Delta^{\mu}$ under Dirichlet boundary conditions.

By Lemma \ref{lem:equalities of measures}, we have that, for all $x \in [a, b]$,
\begin{align*}
\int \mathds{1}_{[a, x]} \cdot 2 \pi n \cdot \cos(2\pi n \cdot F_{\mu})\; \mathrm{d} \mu
&= \int \mathds{1}_{[0, F_{\mu}(x)]}(y) \cdot 2 \pi n \cdot \cos(2\pi n y)\; \mathrm{d} \Lambda(y)\\
&=\sin(2\pi n \cdot F_{\mu}(x))
= f_{n}^{\mu}(x).
\end{align*}
By the definition of the $\mu$-derivative, for all $x \in [a, b]$, we have that
\begin{align*}
f_{n}^{\mu}(x)
= f_{n}^{\mu}(a) + \int \mathds{1}_{[a, x]} \cdot \nabla^{\mu} f_{n}^{\mu} \; \mathrm{d}\mu
= \int \mathds{1}_{[a, x]} \cdot \nabla^{\mu} f_{n}^{\mu} \; \mathrm{d}\mu.
\end{align*}
It therefore follows that $\nabla^{\mu} f_{n}^{\mu} = 2\pi n \cdot \cos(2\pi n \cdot F_{\mu})$.  This together with Lemma~\ref{lem:equalities of measures} implies that, for all $x \in [a, b]$,
\begin{align*}
\nabla^{\mu}f_{n}^{\mu}(x)
&= 2\pi n - 4\pi^{2} n^{2} \int \mathds{1}_{[0, F_{\mu}(x)]}(y) \cdot  \sin(2\pi ny) \; \mathrm{d}\Lambda(y)\\
&= \nabla^{\mu}f_{n}^{\mu}(0) + \int \mathds{1}_{[0, x]} \cdot (- 4\pi^{2} n^{2}) \cdot \sin(2\pi n \cdot F_{\mu}) \; \mathrm{d}\mu,
\end{align*}
and hence that
\begin{align*}
\Delta^{\mu}f_{n}^{\mu}(x) = -4\pi^{2} n^{2} \cdot \sin(2\pi n \cdot F_{\mu}(x)) = -4\pi^{2}n^{2} \cdot f_{n}^{\mu}(x).
\end{align*}
Recall that $\Delta^{\mu}$ is a non-positive operator.  Suppose that $l_{\kappa}$ is an eigenfunction of $\Delta^{\mu}$, with eigenvlaue $\kappa \leq 0$ under Dirichlet boundary conditions.  By \eqref{eq:Delta_Integral_eq} and Lemma \ref{lem:volterra}, it follows that $\nabla^{\mu} l_{\kappa}(a) \neq 0$, and so without loss of generality, we may assume that $\nabla^{\mu} l_{\kappa}(a) = 1$.  Using \eqref{eq:integral equation lebesgue dirichlet} we observe that
\begin{align*}
l_{\kappa} \circ \check{F}_{\mu}^{-1}(x)
&= x + \kappa \int \mathds{1}_{[a, F_{\mu}(x)]}(\check{F}_{\mu}^{-1} \circ F_{\mu}(y)) \cdot (x - F_{\mu} \circ \check{F}_{\mu}^{-1} \circ F_{\mu}(y)) \cdot l_{\kappa} \circ \check{F}_{\mu}^{-1} \circ F_{\mu}(y) \;\mathrm{d}\mu(y)\\
&= x + \kappa \int \mathds{1}_{[a, F_{\mu}(x)]}(\check{F}_{\mu}^{-1}(y)) \cdot (x - F_{\mu} \circ \check{F}_{\mu}^{-1}(y)) \cdot l_{\kappa} \circ \check{F}_{\mu}^{-1}(y) \;\mathrm{d}\mu \circ F_{\mu}^{-1}(y)\\
&= x + \kappa \int \mathds{1}_{[0, x]}(y) \cdot (x - y) \cdot l_{\kappa} \circ \check{F}_{\mu}^{-1}(y) \;\mathrm{d}\Lambda(y).
\end{align*}
For $\alpha > -\kappa$, on the set of continuous functions supported on $[0, 1]$ the operator
\begin{align*}
f(x) \mapsto x + \kappa \int \mathds{1}_{[0, x]}(y) \cdot (x - y) \cdot f(y) \; \mathrm{d}\Lambda(y)
\end{align*}
is contractive with respect to the norm $\lVert \cdot \rVert_{\alpha}$, where the norm $\lVert \cdot \rVert_{\alpha}$ is as defined in the proof of Lemma \ref{lem:volterra}.  Banach's fixed point theorem implies that $l_{\kappa} \circ \check{F}_{\mu}^{-1}(x) = \sin(\sqrt{-\kappa} x) / \sqrt{-\kappa}$.  By the fact that $l_{\kappa}(a) = I_{\kappa}(b) = 0$, it follows that $l_{\kappa} \circ \check{F}_{\mu}^{-1}(0) = l_{\kappa} \circ \check{F}_{\mu}^{-1}(1) = 0$.  Therefore, $l_{\kappa} \circ \check{F}_{\mu}^{-1}$ is an eigenfunction of $\Delta^{\Lambda}$, and hence of the form $x \mapsto \sin(n \pi x)$.
\end{proof}
\section{Examples: Chebyshev Polynomials, Salem functions and Inhomogeneous Cantor maesures}

\subsection{Chebyshev Polynomials as eigenfunctions}\label{sec:Chebyshev}

For this example we choose $a = -b = -1$ and let $\Lambda$ denote the Lebesgue measure restricted to the interval $[-1, 1]$.  Consider the absolutely continuous probability measure $\mu$ supported on $[-1, 1]$ given by
\begin{align*}
\frac{\mathrm{d}\mu}{\mathrm{d}\Lambda} (x) = \frac{1}{\pi \sqrt{1-x^{2}}}.
\end{align*}
The distribution function $F_{\mu}$ can be determined explicitly:
\begin{align*}
F_{\mu}(x)
= \frac{1}{\pi} \int \frac{\mathds{1}_{[-1, x]}(t)}{\sqrt{1-t^{2}}} \; \mathrm{d}\Lambda(t)
= \frac{1}{\pi}\left(\arcsin\left(x\right)+\frac{\pi}{2}\right)
= \frac{1}{\pi} \arccos(-x).
\end{align*}
We remind the reader that the Chebyshev polynomials (of the first kind) are given by
\begin{align*}
T_{0}(x) \coloneqq 1, \quad
T_{1}(x) \coloneqq x, \quad
\text{and} \quad
T_{n+1}(x) \coloneqq 2 x T_{n}(x) - T_{n-1}(x);
\end{align*}
hence $T_{2k}$ is an even function and $T_{2k + 1}$ is an odd function, for $k \in \N_{0}$.  Moreover it is well known that the Chebyshev polynomials can also be defined through the following formula:
\begin{align*}
T_{n} (x) = \cos (n \arccos(x)),\;x\in [-1,1]. 
\end{align*}

This together with Theorem \ref{thm:ev and ef general continuous} implies that the von Neumann eigenfunctions $g_{n}^{\mu}$ are closely related to the Chebyshev polynomials, that is, for all $n \in \N_{0}$ and $x\in [-1,1]$,
\begin{align*}
g_{n}^{\mu}(x)
= \cos( \pi n \cdot F_{\mu}(x))
= \cos( n \cdot \arccos(-x))
= T_{n}(-x)
= (-1)^{n} \cdot T_{n}(x).
\end{align*}
\begin{figure}[h]
\centering
\subfloat[Plot of $g_{0}^{\mu}, g_{1}^{\mu}, g_{2}^{\mu}$ and $g_{3}^{\mu}$.]{
\resizebox{0.45\textwidth}{!}{
\includegraphics{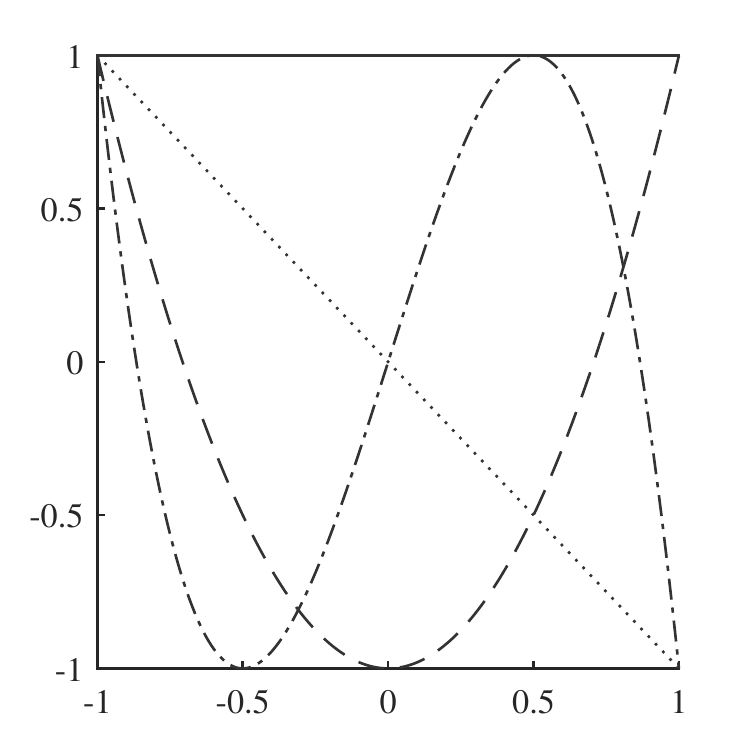}
	}}\hspace{1em}
\subfloat[Plot of $\nabla^{\mu} g_{0}^{\mu}, \nabla^{\mu} g_{1}^{\mu}, \nabla^{\mu}, g_{2}^{\mu}$ and $ \nabla^{\mu} g_{3}^{\mu}$.]{
\resizebox{0.45\textwidth}{!}{
\includegraphics{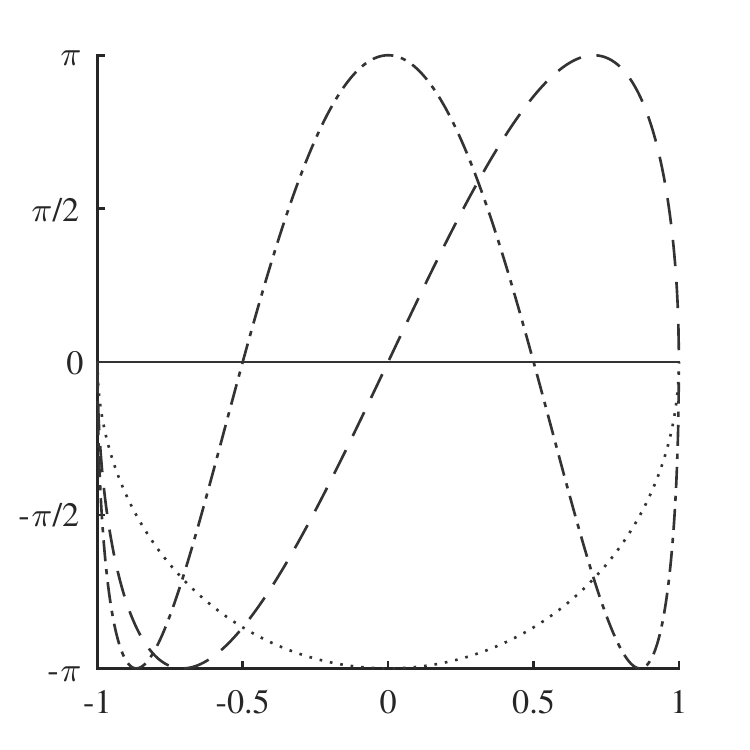}
	}}
\caption{\small Eigenfunctions of $\Delta^{\mu}$, where $\mathrm{d} \mu/ \mathrm{d} \Lambda = 1/(\pi \sqrt{1-x^{2}})$, under von Neumann boundary conditions, and their $\mu$-derivatives.}
\label{Fig1}
\end{figure}
We also observe that, for $n \in \N_{0}$  and $x\in [-1,1]$,
\begin{align*}
\nabla^{\mu} g_{n}^{\mu}(x) = (-1)^{n} \cdot \nabla^{\mu} T_{n}(x) = 
\pi \cdot \sqrt{1 - x^{2}} \cdot \frac{\mathrm{d}T_{n}}{\mathrm{d}x}(x) 
= \frac{(-1)^{n} \cdot \pi \cdot n \cdot ( x \cdot T_{n}(x) - T_{n+1}(x))}{\sqrt{1 - x^{2}}}.
\end{align*}
Moreover, it can be shown that $f_{n}^{\mu}(x) = \nabla^{\mu} g_{n}^{\mu}(x)$, for all $x \in [-1, 1]$ and $n \in \N$.  For further details and results concerning the Chebyshev polynomials we refer the reader to \cite{S92}.

\subsection{Salem measures}

Let us consider a class of examples of distribution functions which was studied by Salem \cite{Sa43}.  Namely we consider the family of absolutely continuous measures $\{ \mu_{p, q} \colon p, q \in (0, 1) \}$, whose distribution functions $\{ F_{\mu_{p, q}} \colon p, q \in (0, 1) \}$ arise from the following endorphisms of $[0, 1]$.  For $r \in (0, 1)$, we define 
\begin{align*}
S_{r}(x) \coloneqq \begin{cases}
x/r & \text{if} \; x \in [0, r],\\
(x - r)/(1-r) & \text{if} \; x \in (r, 1].
\end{cases}
\end{align*}
The maps $F_{\mu_{p,q}} \colon [0, 1] \circlearrowleft$ are then given by $S_{p} \circ F_{\mu_{p, q}} = F_{\mu_{p, q}} \circ S_{q}$.  One can immediately verify that $F_{\mu_{p,q}}$ is strictly monotonically increasing and, for $p \neq q$, is differentiable Lebesgue almost everywhere with derivative equal zero, when it exists.  For more details and further properties of these functions we refer the reader to \cite{JKPS09,KS09,Sa43}.  
\begin{figure}[h]
\centering
\subfloat[Plot of $f_{1}^{\mu}$.]{
\resizebox{0.375\textwidth}{!}{
\includegraphics{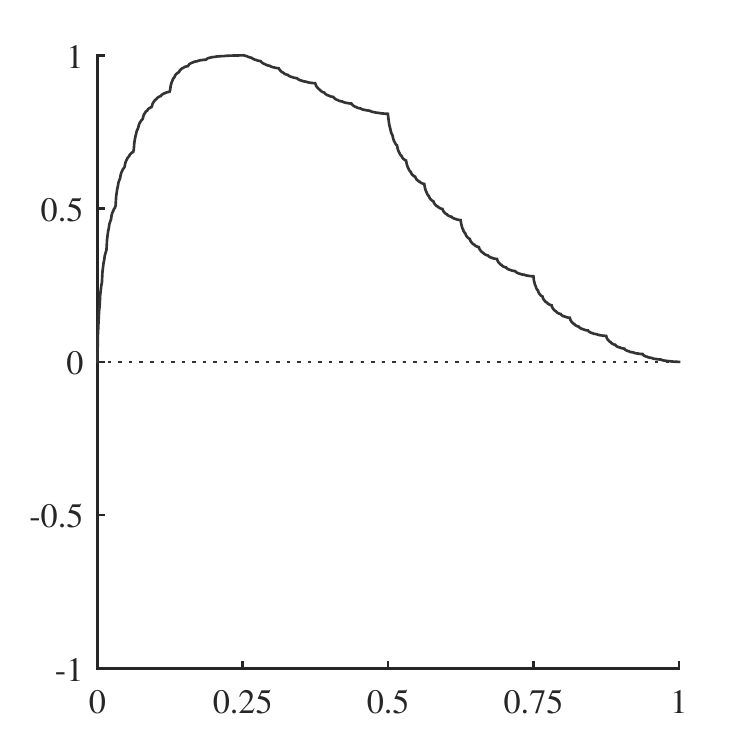}
	}}\hspace{2.5em}
\subfloat[Plot of $f_{2}^{\mu}$.]{
\resizebox{0.375\textwidth}{!}{
\includegraphics{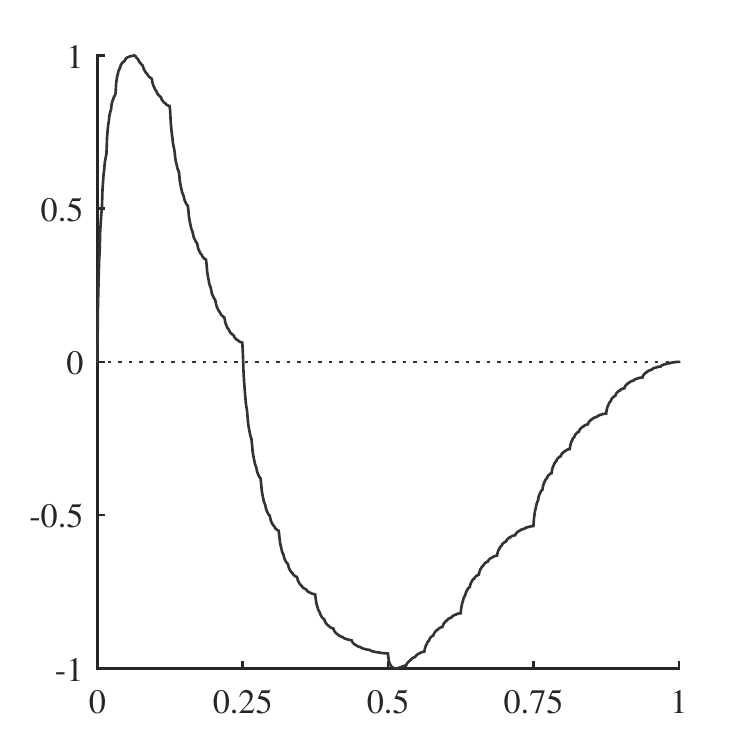}
	}}
	\\
\subfloat[Plot of $f_{3}^{\mu}$.]{
\resizebox{0.375\textwidth}{!}{
\includegraphics{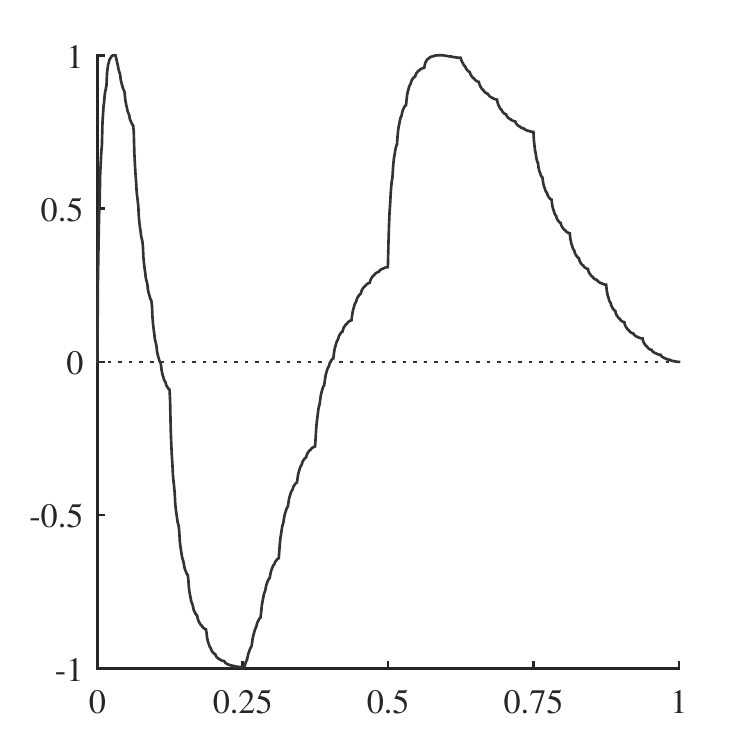}
	}}\hspace{2.5em}
	\subfloat[Plot of $f_{4}^{\mu}$.]{
\resizebox{0.4\textwidth}{!}{
\includegraphics{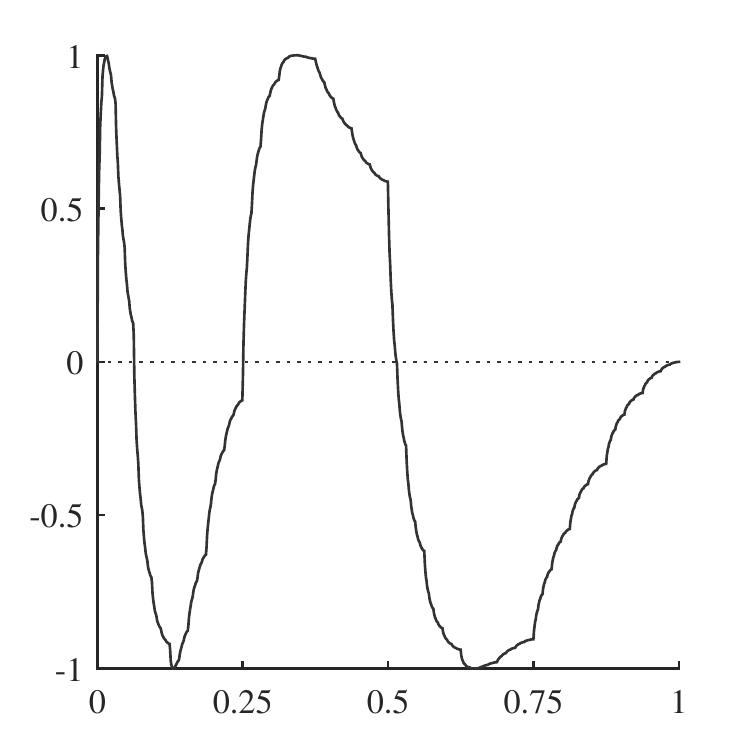}
	}}
\caption{\small Eigenfunctions of $\Delta^{\mu}$, where $\mu = \mu_{p, q}$ is the Salem measure on the unit interval with $p = 0.7$ and $q = 0.5$, under Dirichlet boundary conditions.}
\label{Fig2}
\end{figure}
In Figure \ref{Fig2}, we give graphical representations of the eigenfunctions associated to the eigenvalues $\lambda_{1}, \lambda_{2}, \lambda_{3}, \lambda_{4}$ of the $\mu_{0.7, 0.5}$-Laplacian, under Dirichlet boundary conditions.

\subsection{Inhomogeneous Cantor measures}

Given an interval $[a,b]$ and a finite set of contractions $S = \{s_{i} \colon [a, b] \circlearrowleft \}_{1 \leq i \leq N}$, such that $s_{i}([a, b]) \cap s_{j}([a, b]) = \emptyset$, there exists a unique non-empty set $E \subset [a, b]$ with
\begin{align}\label{eq:self_similar}
E = \bigcup_{i =1}^{N} s_{i}(E).
\end{align}
It is well known that such a set $E$ is homeomorphic to the Cantor set, and, in particular, is totally disconnected.  Further, if $\mathbf{p} = (p_{1}, \dots,p_{N})$ is a probability vector with $p_{i} \in (0,1)$, for all $i \in \{ 1, 2, \dots, N\}$, then there exists a unique atomless Borel probability measure $\mu$ supported on $E$ satisfying
\begin{align*}
\mu(A) = \sum_{i = 1}^{N} p_{i} s_{i}^{-1}(A),
\end{align*}
\begin{figure}[h]
\centering
\subfloat[Plot of $f_{2}^{\mu}$.]{
\resizebox{0.375\textwidth}{!}{
\includegraphics{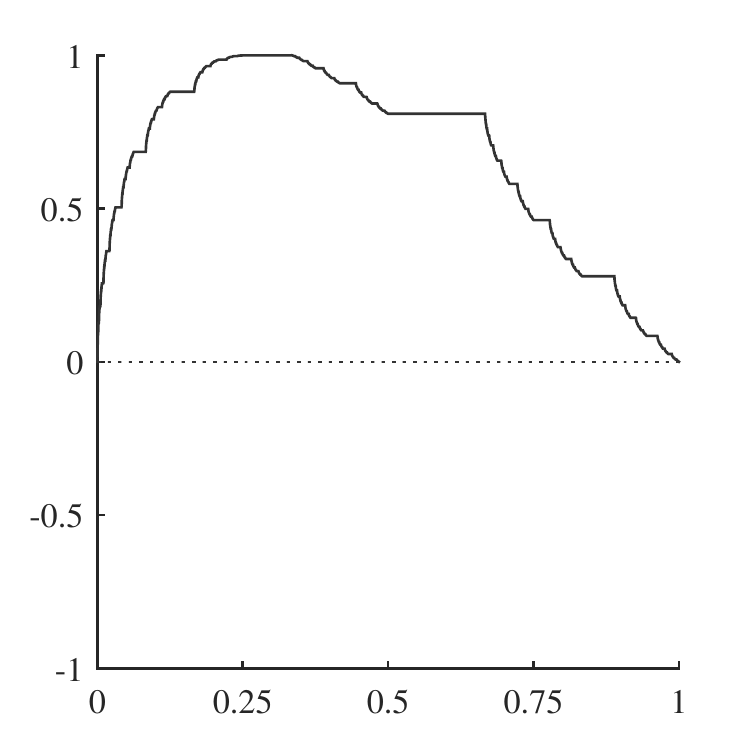}
	}}\hspace{2.5em}
\subfloat[Plot of $f_{4}^{\mu}$.]{
\resizebox{0.375\textwidth}{!}{
\includegraphics{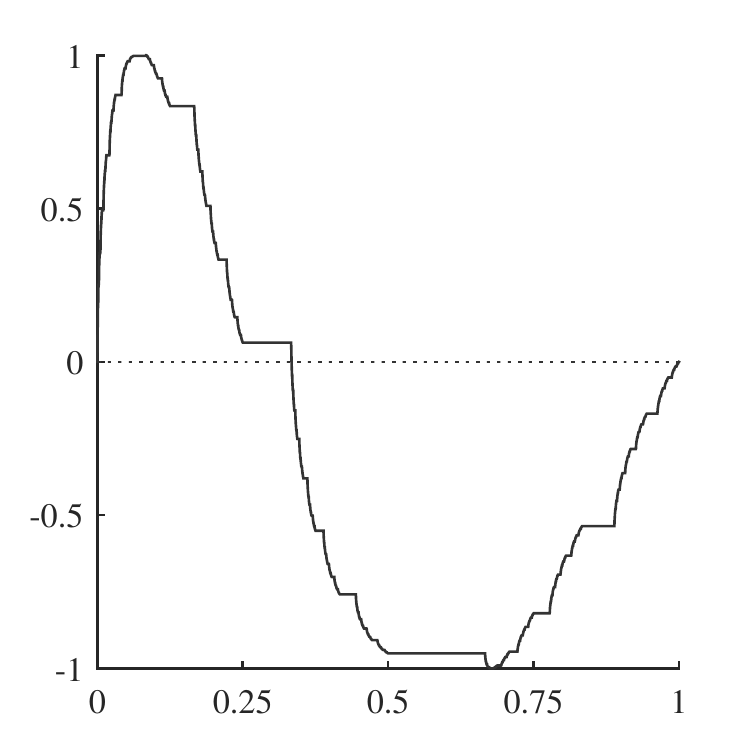}
	}}
	\\
\subfloat[Plot of $f_{6}^{\mu}$.]{
\resizebox{0.375\textwidth}{!}{
\includegraphics{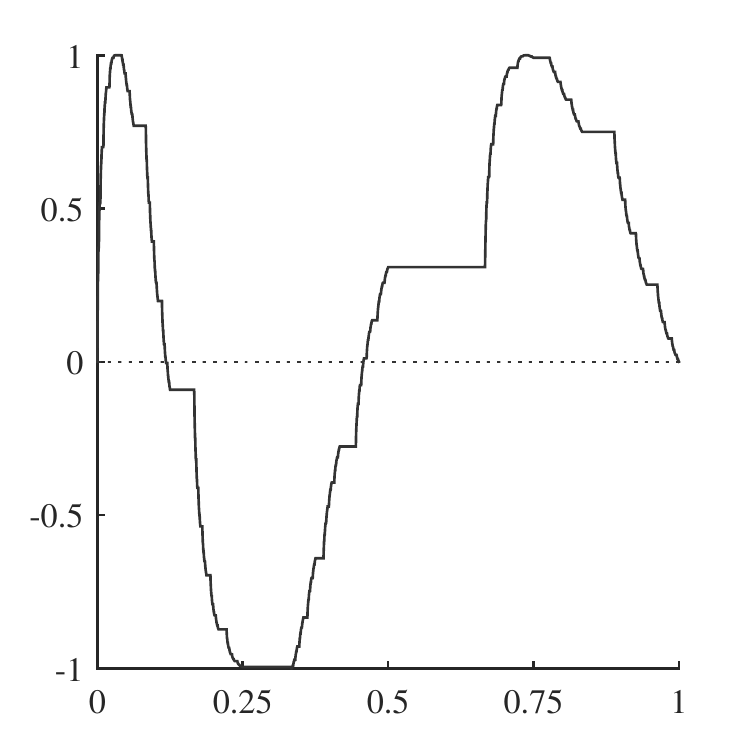}
	}}\hspace{2.5em}
	\subfloat[Plot of $f_{8}^{\mu}$.]{
\resizebox{0.375\textwidth}{!}{
\includegraphics{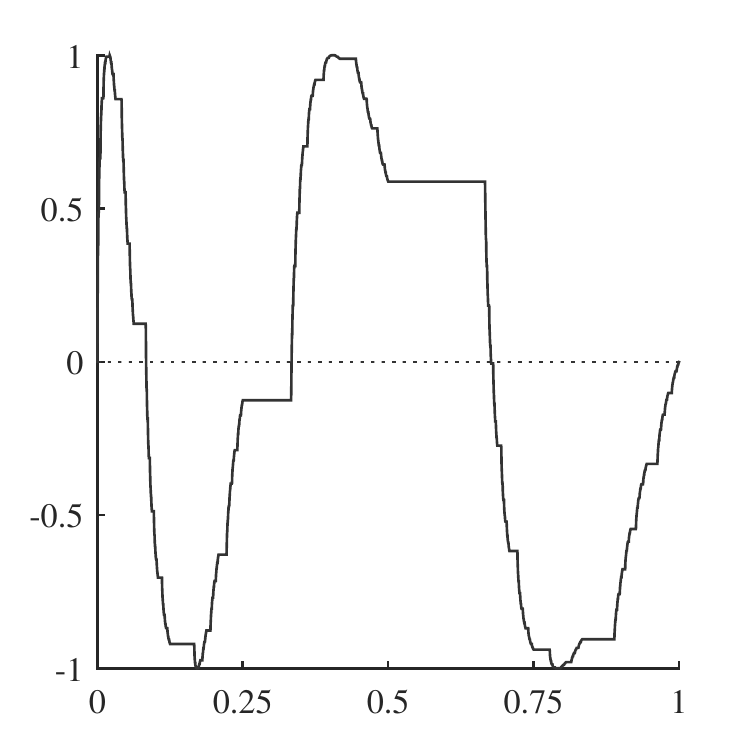}
	}}
\caption{\small Eigenfunctions of $\Delta^{\mu}$, where $[a, b] = [0, 1]$ and $\mu$ is the self-similar measure associated to $S = \{ x \mapsto 0.5 x, x \mapsto x/3 + 2/3 \}$ and $\mathbf{p} = (0.7, 0.3)$, under Dirichlet boundary conditions.
}
\label{Fig3}
\end{figure}
for all Borel measurable sets $A$.  If all of the contractions in $S$ are similarities then the set $E$ is called a \textit{self-similar set} and the measure $\mu$ is called a \textit{self-similar measure}.  Moreover, in this case, if all of the contraction ratios of the members of $S$ are equal, then $E$ is called a \textit{homogeneous self-similar set}; otherwise $E$ is called an \textit{inhomogeneous self-similar set}.  For further properties of the set $E$ and the measure $\mu$, and proof of the above results, we refer the reader to \cite{Fa13,Hu81}.

Let us consider the specific case when $[a, b] = [0, 1]$, $S = \{ x \mapsto 0.5 x, x \mapsto x/3 + 2/3 \}$ and $\mathbf{p} = (0.7, 0.3)$.  Here the unique set $E$ satisfying the equality given in \eqref{eq:self_similar} is an inhomogeneous self-similar set.  Letting $\mu$ denote the associated self-similar measure, in Figures \ref{Fig3}, we give graphical representations of the eigenfunctions associated to the eigenvalues $\lambda_{1}, \lambda_{2}, \lambda_{3}, \lambda_{4}$ of the $\mu$-Laplacian, under Dirichlet boundary conditions.

\end{document}